\documentclass[11pt]{amsart}
\usepackage[T1]{fontenc}
\usepackage{amsfonts}
\usepackage{verbatim}
\usepackage{amsmath, amsthm, amssymb}
\usepackage{natbib}
\usepackage{color}

\numberwithin{equation}{section}
\allowdisplaybreaks

\newtheorem{prop}{Proposition}[section]
\newtheorem{lemma}[prop]{Lemma}

\newtheorem{theorem}[prop]{Theorem}
\newtheorem{defn}[prop]{Definition}

\newtheorem{corollary}[prop]{Corollary}
\newtheorem{proposition}[prop]{Proposition}
\newtheorem{remark}[prop]{Remark}
\newtheorem{example}[prop]{Example}

\newcommand{\reals}{{\mathbb R}}
\newcommand{\bbr}{\reals}

\newcommand{\vep}{\varepsilon}

\newcommand{\BX}{{\bf X}}
\newcommand{\BY}{{\bf Y}}
\newcommand{\BZ}{{\bf Z}}
\newcommand{\BI}{{\bf I}}

\newcommand{\bt}{{\bf t}}

\newcommand{\cI}{\mathcal{I}}

\newcommand{\cC}{\mathcal{C}}

\def\cadlag{c\`adl\`ag}

\begin{document}

\title[Random Locations, Ordered Random Sets and Stationarity]{Random Locations, Ordered Random Sets and Stationarity}

\author[Y. Shen]{Yi Shen}
\address{Department of Statistics and Actuarial Science\\
University of Waterloo\\
Waterloo, ON N2L 3G1, Canada}
\email{yi.shen@uwaterloo.ca}

\thanks{This research was partially supported by NSERC grant.}




\begin{abstract}
Intrinsic location functional is a large class of random locations
containing locations that one may encounter in many cases,
\textit{e.g.}, the location of the path supremum/infimum over a
given interval, the first/last hitting time, \textit{etc}. It has
been shown that this notion is very closely related to stationary
stochastic processes, and can be used to characterize
stationarity. In this paper the author firstly identifies a
subclass of intrinsic location functional and proves that this
subclass has a deep relationship to stationary increment
processes. Then we describe intrinsic location functionals using
random partially ordered point sets and piecewise linear
functions. It is proved that each random location in this class
corresponds to the location of the maximal element in a random set
over an interval, according to certain partial order. Moreover, the
locations changes in a very specific way when the interval of
interest shifts along the real line. Based on these ideas, a
generalization of intrinsic location functional called "local
intrinsic location functional" is introduced and its relationship
with intrinsic location functional is investigated.
\end{abstract}

\maketitle
\section{Introduction}\label{sec: Intro}

\baselineskip=18pt

Stationarity has been an essential concept in stochastic processes
since very long, both due to its theoretical importance and to its
extensive use in modeling. Many related problems,
especially extreme values of stationary processes, have attracted
intensive and ongoing research interests. The classical text
\cite{leadbetter:lindgren:rootzen:1983} and the new book
\cite{lindgren:2012} are both excellent sources for summaries of
existing results and literature reviews. Meanwhile, the random
locations of stationary processes, such as the location of the
path supremum over an interval or the first hitting time of
certain level over an interval, have received relatively less
attention, particularly in a general setting, when the process is
not from one of the few well studied "nice" classes.

In the paper \cite{samorodnitsky:yi:2013b}, the authors introduced
a new notion called "intrinsic location functional", as an
abstraction of the common random locations often considered. More
precisely, let $H$ be a space of real valued functions on $\mathbb
R$, closed under shift. That is, for any $f\in H$ and $c\in \bbr$,
the function $\theta_c f$, defined by $\theta_cf(x)=f(x+c)$, $x\in
\bbr$ is also in $H$. Examples of $H$ include the space of all
continuous functions $\cC({\mathbb R})$, the space of all
c\`adl\`ag functions ${\mathcal D}({\mathbb R})$, the space of all
upper semi-continuous functions, \textit{etc}. Equip $H$ with the
cylindrical $\sigma-$field. Let $\mathcal{I}$ be the set of all
compact, non-degenerate intervals in $\mathbb{R}$:
$\mathcal{I}=\{[a,b]: a<b, [a,b]\subset \mathbb{R}\}$.

\begin{defn}\label{de:ILF}
A mapping $L: H\times \mathcal{I}\to \mathbb{R}\cup\{\infty\}$
is called an intrinsic location functional, if it satisfies the
following conditions.
\begin{enumerate}
\begin{item}
For every $I\in \cI$ the map $L(\cdot, I):\, H\to \bbr \cup\{\infty\}$
is measurable.
\end{item}
\begin{item}
For every $f\in H$ and $I\in\mathcal{I},\,  L(f,I)\in
I\cup\{\infty\}$.
\end{item}
\begin{item}
(Shift compatibility) For every $f\in H$, $I\in\mathcal{I}$ and
$c\in\mathbb{R}$,
$$
L(f,I)=L(\theta_c f, I-c)+c,
$$
where $I-c$ is the interval $I$ shifted by $-c$, and $\infty+c=\infty$.
\end{item}
\begin{item}
(Stability under restrictions)  For every $f\in H$ and $I_1, I_2\in
\mathcal{I}$, $I_2\subseteq I_1$,
$$
\text{if} \ \ L(f,I_1)\in I_2, \  \text{then} \ \  L(f, I_2)=L(f, I_1).
$$
\end{item}
\begin{item}
(Consistency of existence)  For every $f\in H$ and $I_1, I_2\in
\mathcal{I}$, $I_2\subseteq I_1$,
$$
\text{if} \ \ L(f,I_2)\neq\infty, \  \text{then} \ \
 L(f,I_1)\neq\infty.
$$
\end{item}
\end{enumerate}
\end{defn}

It is not difficult to realize that intrinsic location functional
is an abstraction of common random locations such as the location
of the path supremum/infimum over an interval, the first/last
hitting time over an interval, among many others. Interested
readers are invited to see \cite{samorodnitsky:yi:2013b} for more
examples and counterexamples of intrinsic location functionals.
Notice that in the definition we included $\infty$ as a possible
value. This corresponds to the fact that not all the random
locations are necessarily well-defined for all the paths. For
instance, a path can lie above certain level over the whole
interval of interest, leaving the first/last hitting time
undefined. Here and later, we always assign $\infty$ as the value
of an intrinsic location functional when it is otherwise
undefined. Accordingly, the $\sigma-$field used for ${\mathbb R}\cup\{\infty\}$ is generated by the Borel $\sigma-$field plus $\infty$ as a singleton.

It turns out that, despite the huge variety of the origins and
natures of these random locations, the common points that they
share, now summarized in the definition of intrinsic location
functional, are sufficient to guarantee many interesting and
important properties of their distributions for stationary
processes. The majority of these properties are firstly studied in \cite{samorodnitsky:yi:2012} and \cite{samorodnitsky:yi:2013a}, for the location of path supremum over compact intervals.

Fix a path space $H$. Let us denote the stochastic process by
$\BX$, with all sample paths in $H$, and the intrinsic location
functional by $L$. Then for each fixed interval $I=[a,b]\in\cI$,
$L(\BX,I)$ is a random variable taking value on $\mathbb
R\cup\{\infty\}$. Denote its cumulative distribution function by
$F_{\BX,I}$ or $F_{\BX,[a,b]}$. When the stationarity is assumed,
it is clear that the location of the interval $I$ will not
affect the distribution of $L(\BX,I)-a$, as long as the length of
the interval, $|I|=b-a$, remains constant. In this case we often
fix the starting point $a$ to be $0$, and use the shorter notation
$F_{\BX,b}$.

\begin{theorem}\label{t:density.rcll}[\cite{samorodnitsky:yi:2013b}]
Let $L$ be an intrinsic location functional and $\BX =
(X(t),\, t\in\bbr)$  a stationary process. Then the restriction of the
law
$F_{\BX,T}$ to the interior $(0,T)$ of the interval is absolutely
continuous. The density, denoted by $f_{\BX,T}$, can be taken to
be equal to the right derivative of the cdf $F_{\BX,T}$, which
exists at every point in the interval $(0,T)$. In this case the
density is right continuous, has left limits, and has the
following properties.

(a) \ The limits
$$
f_{\BX,T}(0+)=\lim_{t\to 0} f_{\BX,T}(t) \ \text{and} \
f_{\BX,T}(T-)=\lim_{t\to T} f_{\BX,T}(t)
$$
exist.

(b) \ The density has a universal upper bound given by
\begin{equation} \label{e:density.bound}
f_{\BX,T}(t) \leq \max\left(\frac1t,\frac{1}{T-t}\right), \
0<t<T\,.
\end{equation}

(c) \ The density has a bounded variation away from the endpoints
of the interval. Furthermore, for every $0<t_1<t_2<T$,
\begin{equation} \label{e:TV.away}
TV_{(t_1, t_2)}(f_{\BX,T}) \leq \min\bigl( f_{\BX,T}(t_1),
f_{\BX,T}(t_1-)\bigr)  + \min\bigl( f_{\BX,T}(t_2), \,
f_{\BX,T}(t_2-)\bigr)\,,
\end{equation}
where
$$
TV_{(t_1, t_2)}(f_{\BX,T}) = \sup\sum_{i=1}^{n-1}
\bigl|f_{\BX,T}(s_{i+1})- f_{\BX,T}(s_i)\bigr|
$$
is the total variation of $f_{\BX,T}$ on the interval $(t_1,t_2)$,
and the supremum is taken over all choices of
$t_1<s_1<\ldots<s_n<t_2$.

(d) \ The density has a bounded positive variation at the left
endpoint and a bounded negative variation at the right endpoint.
Furthermore, for every $0<\vep<T$,
\begin{equation} \label{e:TV.left}
TV^+_{(0,\vep)}(f_{\BX,T}) \leq \min\bigl( f_{\BX,T}(\vep),
f_{\BX,T}(\vep-)\bigr)
\end{equation}
and
\begin{equation} \label{e:TV.right}
TV^-_{(T-\vep,T)}(f_{\BX,T}) \leq \min\bigl( f_{\BX,T}(T-\vep),
f_{\BX,T}(T-\vep-)\bigr)\,,
\end{equation}
where for any interval $0\leq a<b\leq T$,
$$
TV^\pm_{(a,b)}(f_{\BX,T}) = \sup\sum_{i=1}^{n-1}
\big(f_{\BX,T}(s_{i+1})- f_{\BX,T}(s_i)\bigr)_\pm
$$
is the positive (negative) variation of $f_{\BX,T}$ on the
interval $(a,b)$, and the supremum is taken over all choices of
$a<s_1<\ldots<s_n<b$.

(e) \ The limit $f_{\BX,T}(0+)<\infty$ if and only if
$TV_{(0,\vep)}(f_{\BX,T}) <\infty$ for some (equivalently, any)
$0<\vep<T$, in which case
\begin{equation} \label{e:TV.left1}
TV_{(0,\vep)}(f_{\BX,T}) \leq f_{\BX,T}(0+)+\min\bigl(
f_{\BX,T}(\vep), f_{\BX,T}(\vep-)\bigr)\,.
\end{equation}
Similarly, $f_{\BX,T}(T-)<\infty$ if and only if
$TV_{(T-\vep,T)}(f_{\BX,T}) <\infty$ for some (equivalently, any)
$0<\vep<T$, in which case
\begin{equation} \label{e:TV.right1}
TV_{(T-\vep,T)}(f_{\BX,T}) \leq \min\bigl( f_{\BX,T}(T-\vep),
f_{\BX,T}(T-\vep-)\bigr) + f_{\BX,T}(T-)\,.
\end{equation}
\end{theorem}

The key properties in this theorem, (c), (d) and (e), are called "total variation constraints", since they put constraints on the total variation of the density functions. It was then proved that the total variation constraints of the intrinsic location functionals are not merely a group of properties of stationary processes: they are actually the stationarity itself, viewed from a different angle.

\begin{theorem}\label{t:equivalence}[\cite{samorodnitsky:yi:2013b}]
Let $\BX$ be a stochastic process with continuous sample paths. The
following statements are equivalent.
\begin{enumerate}
\begin{item}
The process $\BX$ is stationary.
\end{item}
\begin{item}
For some (equivalently, any) $\Delta>0$,
any intrinsic location functional $L:\, C(\bbr)\times
\mathcal{I}\to \mathbb{R}\cup\{\infty\}$, the law of $L(\BX, I)-a$,
$I=[a,a+\Delta]\in \mathcal{I}$, does not depend on $a$.
\end{item}
\begin{item}
For any intrinsic location
functional $L:\, C(\bbr)\times \mathcal{I}\to
\mathbb{R}\cup\{\infty\}$, any interval $I=[a,b]\in
\mathcal{I}$, the law of $L(\BX, I)$ is absolutely continuous on
$(a,b)$ and has a density satisfying the total variation
constraints.
\end{item}
\end{enumerate}
\end{theorem}

To sum up, the notion of intrinsic location functional has been
introduced, and its deep relationship to the stationarity has been
revealed. It can even be used as an alternative definition of
stationarity.

On the other hand, there remain very important questions to ask.
Firstly, are there similar results for larger families of
stochastic processes compared to stationary processes? The set of
stationary increment processes, for instance, includes all the
stationary processes, but also many commonly used non-stationary
processes, such as Brownian motion or L\'{e}vy processes in
general. What properties do the distributions of random locations
of these processes have? Secondly, there has not been many results
developed to describe the object of intrinsic location functional,
therefore it is also interesting to proceed in this direction.
Some representation results, for example, will also be very
valuable.

In this paper, we will answer the questions in these two
directions. A subclass of intrinsic location functionals, called
``doubly intrinsic location functionals'', will be identified, and
its deep relation with stationary increment processes will be
investigated. For the other direction, we develop equivalent
descriptions, as well as an important generalization, of intrinsic
location functionals. These new results will be highly helpful for
a better and more comprehensive understanding of the notion of
intrinsic location functional.

The rest of the paper is organized in the following way. In part
two we define the ``doubly intrinsic location functionals'', and
show that this subclass of intrinsic location functional can be
used to fully characterize the stationarity of the increments of a
process. In part three, a generalization of intrinsic location
functional called "local intrinsic location functional" is
introduced, which allows one to define a random location only for
intervals with a single fixed length. Then we develop descriptions
for it and also for intrinsic location functionals using partially
ordered random point sets. The relation between local intrinsic
location functional and intrinsic location functional is
investigated in part four, showing that the former naturally
inherits most of the properties of the latter. We provide yet
another description in part five, which focuses on characterizing
the value of a (local) intrinsic location functional as a function
of the location of the interval of interest when the length
of the interval is fixed.

\section{Random locations of stationary increment processes}

Certain intrinsic location functionals, such as the location of the path supremum/infimum
over an interval, the hitting times of the derivative of the path assuming it is $C^1$, possess
the property of ``vertical shift invariance'', in the sense that their values will not change when the path
is shifted vertically. In order to benefit from this additional property, we add the
vertical shift invariance to the definition to form the new notion
of ``doubly intrinsic location functional''.

\begin{defn}
An intrinsic location functional $L$ is called doubly intrinsic,
if for every function $f\in H$, every interval $I\in\mathcal{I}$ and
every $c\in\mathbb{R}$,\index{intrinsic location functional!doubly intrinsic location functional}
$$
L(f,I)=L(f+c,I).
$$
Denote by $\mathcal{D}$ the set of all doubly intrinsic location
functionals defined on $H$.
\end{defn}

The word ``doubly'' in the name refers to the fact that $L$ is
both ``horizontally shift compatible'', in the sense that it moves
along with the function and the interval horizontally, and
``vertically shift invariant'', in the sense that it does not move
along with the function vertically.

In general, once we verify that certain location is an intrinsic
location functional, it is very easy to check whether it is doubly
intrinsic or not. Intuitively, an intrinsic location functional is
doubly intrinsic if and only if its value only depends on the
``shape'' of the function and does not depend on the ``height'' of
the function. Here are some most natural and important examples of
doubly intrinsic location functionals.

\begin{example}{\rm
Let $H$ be the space of all the upper (lower) semi-continuous
functions. Then the location of the path supremum (infimum) over an
interval $I$,
$$
\tau_{f,I}:=\inf\{t\in I: f(t)=\sup(\inf)_{s\in I}f(s)\}
$$
is a doubly intrinsic location functional. The infimum outside
means that in case of a tie, we always chose the leftmost point
among all the points achieving the path supremum (infimum).}
\end{example}

\begin{example}{\rm
Let $H$ be the space of all c\`adl\`ag functions. Then the
time of the first jump in a period $[a,b]$,\index{firstjump}
$$
T^\Delta_{f,[a,b]}:=\inf\{t\in[a,b], f(t-)\neq f(t)\}
$$
is a doubly intrinsic location functional.}
\end{example}

Needless to say, any random location which only depends on the
value of the first derivative of $\mathcal{C}^1$ functions is also
doubly intrinsic. For instance, the location of the first local
maxima, the first time that the derivative hits certain level,
\textit{etc}. The class of doubly intrinsic location functionals
extends, however, far beyond these ``natural'' examples. Actually,
let $H$, $H'$ be two spaces of functions, and $\varphi$ be a
mapping from $H$ to $H'$ which is interchangeable with
translation:
\begin{equation}\label{e:trans}
\forall f\in H, \forall c\in\mathbb{R}, \varphi(\theta_c
f)=\theta_c (\varphi f),
\end{equation}
and consistent with vertical shift:
\begin{equation}\label{e:vshift}
\forall f\in H, \forall c\in\mathbb{R}, \exists c'\in\mathbb{R},
\varphi(f+c)=\varphi(f)+c'.
\end{equation}
If $L'$ is a doubly
intrinsic location functional on $H'\times \mathcal{I}$, then the
functional $L$ on $H\times \mathcal{I}$, defined by
$$
L(f,I):=L'(\varphi f, I), \quad \forall f\in H, \forall I\in
\mathcal{I},
$$
is also a doubly intrinsic location functional, provided that the measurability condition is satisfied. We call it the
doubly intrinsic location functional induced by $\varphi$. This
procedure allows us to associate random locations which are
originally only well-defined for ``nice'' functions to the
functions which does not possess the required properties. The
transforms satisfying (\ref{e:trans}) and (\ref{e:vshift}) include
many commonly used operations such as convolution with a given function, differentiation,
moving average, moving difference, \textit{etc}.

\begin{example}\label{e:mollifier}{\rm
Let $\psi$ be the classical mollifier:
$$
\psi(x)=\left\{\begin{array}{ll} e^{-1/(1-|x|^2)} & \text{ if }
|x|<1\\
0 & \text{ if } |x|\geq 1
\end{array}\right . ,
$$
then the operation of convolution with $\psi$ transforms any
measurable function to a smooth function. That is, let $f$ be any
measurable function, then $f*\psi$ is a smooth function, where
``$*$'' denotes convolution. This convolution is obviously
interchangeable with translation. It is easy to see that the
location of the first hitting time of the derivative to level $h$
over an interval:
$$
L'(g,I):=\inf\{t\in I: g'(t)=h\}
$$
(following the tradition that $\inf\phi=\infty$) is a doubly
intrinsic location functional on the space of all smooth
functions.

We will call
a set $H$ of functions on $\mathbb{R}$ a {\it LI} set\index{LI set} (from {\it
locally integrable}) if it has following properties:
\begin{itemize}
\item $H$ is invariant under shifts; \item $H$ is equipped with
its cylindrical $\sigma$-field $\cC_H$; \item the map
$H\times\bbr\to\bbr$ defined by $(f,t)\to f(t)$ is
  measurable;
\item any $f\in H$ is locally integrable.
\end{itemize}
An example of {\it LI} set is the space $\mathcal{D}(\bbr)$ of \cadlag\
functions on $\bbr$. Note that, by Fubini's theorem, for any
{\it LI} set $H$, the map $T_\psi:\, H\to
\mathcal{C}(\bbr)$, defined by
\begin{equation} \label{e:the.map}
T_\psi(f) = f*\psi=\int_{-\infty}^\infty f(s)\psi(t-s)\, ds, \ t\in\bbr
\end{equation}
is $\cC_H/\cC_{\mathcal{C}(\bbr)}$-measurable. Therefore, if, moreover, the space $H$ in this example is a $LI$ set, then the measurability issue for the
induced location functional
$$
L(f,I):=L'(f*\psi,I)
$$
is guaranteed. Thus $L$ is also a doubly intrinsic location functional, now defined on any $LI$ set. The doubly intrinsic location
functionals of this kind will play an important role in the
proof of the theorem below.}
\end{example}

\begin{theorem}\label{t:sta.increment}
Let $\BX$ be a stochastic process having path in $H$ with
probability 1, where $H$ is a LI set. Then the followings are equivalent.
\begin{enumerate}
\begin{item}
The process $\BX$ is of stationary increments.
\end{item}
\begin{item}
For some (equivalently, any) $\Delta>0$, any doubly intrinsic
location functional $L: H\times \mathcal{I} \to {\mathbb
R}\cup\{\infty\}$, the law of $L(\BX, I)-a$,
$I=[a,a+\Delta]\in{\mathcal I}$, does not depend on $a$.
\end{item}
\begin{item}
For any doubly intrinsic location functional $L: H\times
\mathcal{I} \to {\mathbb R}\cup\{\infty\}$, any interval
$I=[a,b]\in\mathcal{I}$, The law of $L(\BX, I)$ is absolutely
continuous on $(a,b)$ and has a density satisfying the total
variation constraints (\ref{e:TV.away})-(\ref{e:TV.right1}) .
\end{item}
\end{enumerate}
\end{theorem}

Similar to the case of intrinsic location functionals and
stationary processes, this theorem shows that there is a deep and
fundamental relationship between the stationarity of increments,
the shift invariance of the distributions of doubly intrinsic
locations, and the total variation constraints. The most
surprising part is that the total variation constraints alone are
enough to imply the stationarity of increments, even there is no
distributional invariance explicitly formulated at all.
Intuitively, it seems to be totally possible that all the doubly
intrinsic location functionals always satisfy the total variation
constraints, yet their distributions change over different period.
This theorem, however, tells us that this will never happen. The
total variation constraints automatically lead to the distributional
invariance under translation. It could be the case that for some doubly intrinsic
location functional, its distribution varies over time while
always keeping the total variation constraints obeyed; but then
there must be some other doubly intrinsic location functional, for
which the total variation constraints are violated. As a family of
random locations, the doubly intrinsic location functional is rich
enough such that the total variation constraints on this family
provide enough information to guarantee the stationarity of the
increment of the process.

It is also interesting to make a comparison between Theorem
\ref{t:sta.increment} and its stationary counterpart, Theorem
\ref{t:equivalence}. In each of these cases, we have two spaces:
the space of processes and the space of location functionals. In
Theorem \ref{t:equivalence}, the space of processes is the
stationary processes, and the corresponding space of location
functionals is the intrinsic location functionals. The two spaces
are related one to each other via the total variation constrains.
In this sense, the total variation constraints introduce a
``duality'' between the space of processes and the space of random
locations. In Theorem \ref{t:sta.increment}, the space of
processes becomes the stationary increment processes. Notice that
since stationary processes are automatically of stationary
increments but the converse is not true, the space of stationary
increment processes is strictly larger than the space of
stationary processes. Therefore we should expect a smaller space
of the locations on the other side of the duality. It is indeed
the case here, since doubly intrinsic location functionals is by
definition a proper subset of intrinsic location functionals. In
conclusion, Theorem \ref{t:sta.increment} and Theorem
\ref{t:equivalence} have the same nature, but are with different sizes of
the sets on both sides of the duality.

Now let us turn to the proof of Theorem \ref{t:sta.increment}. The
proof actually highly resembles the corresponding proof of Theorem
\ref{t:equivalence} presented in \cite{samorodnitsky:yi:2013b}.
The full proof will have four directions: $(1)\to(2)$, $(1)\to
(3)$, $(2)\to (1)$ and $(3)\to (1)$. Given the fact that the
proofs for some directions are very long, we will not include
everything in the proof below, but will refer to the same proofs
in \cite{samorodnitsky:yi:2013b} when it is possible. Many lemmas
and settings, however, require changes and reverification.

First of all, notice the following lemma:

\begin{lemma} \label{l:in.location.var}
Let $\BX$ be a stationary increment process with paths in $H$
almost surely. Let $L\in\mathcal{D}$ and denote by $F_{\BX,
I}(\cdot)$ the distribution of $L(\BX,I)$. Then\\
\indent {\rm (i)} \ For any $\Delta\in\bbr$,
$$
F_{\BX,[\Delta,T+\Delta]}(\cdot) = F_{\BX,[0,T]}(\cdot-\Delta)\,.
$$

{\rm (ii)} \ For any intervals $[c,d]\subseteq [a,b]$,
$$
F_{\BX,[a,b]}(B)\leq F_{\BX,[c,d]}(B) \ \text{for any Borel set
$B\subseteq [c,d]$.}
$$

{\rm (iii)} \ For any intervals $[c,d]\subseteq [a,b]$,
$$
F_{\BX,[a,b]}(\{\infty\})\leq F_{\BX,[c,d]}(\{\infty\}).
$$
\end{lemma}

\begin{proof}
The point (ii) and (iii) are direct results of the stability under
restriction and the consistency of existence in the definition of
intrinsic location functionals, respectively. For (i), define process $\BY(t):=\BX(t)-\BX(\Delta)+\BX(0),
t\in{\mathbb R}$, then the stationarity of the increments implies
that the process $\BY(\cdot+\Delta)$ has the same distribution as
$\BX(\cdot)$. Thus
$$
F_{\BX,[\Delta, T+\Delta]}(\cdot)=F_{\BY, [0,T]}(\cdot).
$$
Although $\BY(t)-\BX(t)=\BX(0)-\BX(\Delta)$ is random and depends
on the realization, it is a constant over time. Thus
$$
L(\BX,[0,T])=L(\BY,[0,T]),
$$
hence
$$
F_{\BX,[0,T]}(\cdot)=F_{\BY, [0,T]}(\cdot).
$$
\end{proof}

The rest of the proof in the direction $(1)\to(2)$ and $(1)\to(3)$
follows in the same way as in \cite{samorodnitsky:yi:2013b}.

To prove that $(2)\to (1)$, consider the following location
functional:
$$
G_{\bt,\BI}(\BX, [a,a+\Delta]):=\inf\{t\in[a,a+\Delta]: t\in
S(\BX,\bt,\BI)\},
$$
where the random set of points $S$ is defined by
$$
S(\BX,\bt,\BI):=\{t\in\mathbb{R}: \BX(t+t_i)-\BX(t)\in I_i,
\forall i=1,...,n\},
$$
$n$ is a positive integer, $\bt=(t_1,...,t_n)$ such that
$0<t_1<...<t_n$, and $\BI=I_1\times...\times I_n\in\mathcal{I}^n$.
It is then easy to check that such defined $G_{\bt,\BI}$ is a
doubly intrinsic location functional for any $n=1,2,...$, any
$\bt$ and $\BI$. Moreover, $G_{\bt,\BI}(\BX, [a,a+\Delta])=a$ if
and only if
$$
\BX(a+t_i)-\BX(a)\in I_i, \forall i=1,...,n.
$$
If the distribution of $G_{\bt,\BI}$ does not depend on $a$, the
probability that $ \BX(a+t_i)-\BX(a)\in I_i, \forall i=1,...,n. $
can not depend on $a$. Since this shift invariance holds for all
$n, \bt$ and $\BI$, the stationarity of the increments is
guaranteed.

We are now left with the proof that $(3)\to (1)$. The main object
that we are going to consider are the doubly intrinsic location
functionals of the type of Example \ref{e:mollifier}, but slightly
more complicated. More precisely, let the function $\psi$ as
defined in Example \ref{e:mollifier}. Define process $\BY:=\BX *
\psi$, then $\BY$ is a stationary increment process with smooth
path. Consequently, $\BZ=\BY'$, the derivative of $\BY$, is a
smooth stationary process. For any $n=1,2,...$, any $h>0$, $d\geq
0$, any $\bt=(t_0,t_1,...,t_n)$ such that $0<t_0<t_1<...<t_n$ and
any $\BI=I_1\times ... \times I_n \in \mathcal{I}^n$, define the
random set of points
$$
A^{h,d}_{\bt, \BI}(\BX)=\{s\in \mathbb{R}: \BZ(s)=h, \inf\{r>s:
\BZ(r)=h\}> t+d,
$$
$$
\BX(s+t_i)-\BX(s+t_0)\in I_i, \forall i=1,...,n\}.
$$
Notice that the $LI$ setting guarantees the measurability. This
set seems to be a little strange at the first glance, since the
points are marked according to the process $\BZ$, but then
filtered using conditions on the original process $\BX$. However,
since $\BZ$ is transformed from $\BX$ and both the operation of
convolution and differentiation are interchangeable with
translation, the location
$$
L(\BX,I):=\inf\{t: t\in A^{h,d}_{\bt,\BI}(\BX)\cap I\}
$$
is an intrinsic location functional. Moreover, since the points
are marked on the derivative $\BZ$ and then filtered using
conditions only on the increments $\BX(s+t_i)-\BX(s+t_0)$, the
location $L(\BX,I)$ is invariant under vertical shift. Hence
$L(\BX,I)$ is a doubly intrinsic location functional. After
defining
\begin{equation}\label{e:def.p}
p^{h,d}_{\bt,\BI,a,\Delta}(\BX)=\mathbb{P}(A^{h,d}_{\bt,
\BI}(\BX)\cap [a, a+\Delta]\neq \phi),
\end{equation}
we are totally back to the track of the proof for the stationary
case (Theorem \ref{t:equivalence}, proved in
\cite{samorodnitsky:yi:2013b}). Here we list the corresponding
forms that the lemmas should take under the stationarity of
increments.

\begin{lemma}\label{l:in.fix.level} Let $\BX$ be a stochastic process.
If condition (3) in Theorem \ref{t:sta.increment} is satisfied, then for any $h,d,\bt$ and $\BI$ defined as
before, with probability 1, $A^{h,d}_{\bt, \BI}(\BX)$ is either
the empty set or an infinite set, in which case
$\inf(A^{h,d}_{\bt,\BI}(\BX))=-\infty$ and
$\sup(A^{h,d}_{\bt,\BI}(\BX))=\infty$.
\end{lemma}

\begin{lemma}\label{l:in.points.sta}
Given $h\in\mathbb{R}$, if for any $\Delta >0$,$d\geq 2\Delta$,
$\bt$ and $\BI$ defined as before,
$p^{h,d}_{\bt,\BI,a,\Delta}(\BX)$ is always constant on $a$, then
the process $\BX$ is of stationary increments.
\end{lemma}

\begin{lemma}\label{l:in.pconst}
Assume that for any doubly intrinsic location functional
$L\in\mathcal{D}$, any interval $I\in\mathcal{I}$, $L(\BX, I)$
admits a density function $f_{\BX,I}(t)$ in $\mathring{I}$, which
satisfies the total variation constraints on $I$. Then
$p^{h,d}_{\bt,\BI,a,\Delta}(\BX)$ is constant on $a$ for any
$\Delta >0$, $d\geq 2\Delta$, $\bt$ and $\BI$ defined as before.
\end{lemma}

Lemma \ref{l:in.fix.level} gives us the right to decompose the
path space and focus on only one given $h$. Lemma
\ref{l:in.pconst} and Lemma \ref{l:in.points.sta} then lead to the
desired result in a straightforward way.

\section{Definition of local intrinsic location functional and representation by ordered set}

The results reviewed in Section 1 showed how closely the concept of
intrinsic location functional is related to stationarity. In some
sense, the total variation constraints for intrinsic location
functionals are just stationarity itself viewed from a different
perspective. However, if one only considers the total variation
constraint for intervals with a particular length, condition (4)
and (5) in Definition \ref{de:ILF} may
appear unnecessarily restrictive: in order to get the total
variation constraint for the intervals with this length, one needs
to introduce the relationships between intervals with all different
lengths. Therefore it is interesting to check if we can adjust
the definition of intrinsic location functional, so that it can be defined only for intervals with the given length, while assuring
that the total variation constraints still hold for the intervals with this
length. It turns out that a reasonable way for this purpose
is to define the following object, which we name as ``local intrinsic
location functional''.

\begin{defn}\label{defn:lilf}
Let $T>0$ be given. A mapping $L_T: H\times \mathbb{R}\to \mathbb{R}\cup\{\infty\}$
is called a local intrinsic location functional with related length $T$, if it satisfies the
following conditions.
\begin{enumerate}
\begin{item}
For every $a\in {\mathbb R}$, the map $L_T(\cdot, a):\, H\to \bbr \cup\{\infty\}$
is measurable.
\end{item}
\begin{item}
For every $f\in H$ and $a\in\mathbb{R},\,  L_T(f,a)\in
[a,a+T]\cup\{\infty\}$.
\end{item}
\begin{item}
For every $f\in H$, $a\in\mathbb{R}$ and
$c\in\mathbb{R}$,
$$
L_T(f,a)=L_T(\theta_c f, a-c)+c,
$$
where $\infty+c=\infty$.
\end{item}

\begin{item}
For every $f\in H$ and $a,b\in {\mathbb R}$,
$L_T(f,a)\in[b,b+T]$ implies that either
$L_T(f,b)=L_T(f,a)$, or $L_T(f,b)\in[b,b+T]\backslash[a,a+T]$.
\end{item}
\end{enumerate}
\end{defn}

The first three conditions are the same as in
the definition of intrinsic location functional. The condition (4)
is new and replaces both condition (4) and (5) in Definition
\ref{de:ILF}. Intuitively, it first requires that if the locations
for two intervals with the same length both fall into the
intersection of these two intervals, then they must agree. This is
a counterpart of condition (4) (stability under restriction) in
Definition \ref{de:ILF}, but now only
explicitly involving intervals with one fixed length. The second
possibility in condition (4) says that if the location for the
first interval is located in the second interval yet is no longer
the corresponding location for the second interval, then it must
be replaced by another point which is located in the second
interval but outside the first interval. In particular, the
corresponding location for the second interval can not take value
$\infty$. In this sense, the second part of condition (4) actually
serves as an alternative of condition (5) (consistency of
existence) in Definition \ref{de:ILF}.

It is not difficult to see that if we restrict the definition of
an intrinsic location functional to intervals with a fixed length,
then it automatically gives out a local intrinsic location
functional:

\begin{example}{\rm
Let $L: H\times {\mathcal I}\to {\mathbb R}\cup\{\infty\}$ be an intrinsic location functional. Then it is easy to check that for any fixed length $T>0$, $L_T$ defined by
$$
L_T(f,a)=L(f,[a,a+T]),
$$
$f\in H, a\in{\mathbb R}$ is a local intrinsic location
functional.}
\end{example}

On the other hand, a natural ``extension'' of a
local intrinsic location functional to intervals with different
lengths does not necessarily give out an intrinsic location
functional, as shown by the following example.

\begin{example}{\rm
Let $H=\mathcal C(\mathbb R)$, $l>0$, $L_T(f,a)$ be the first hitting
time to a fixed level $h$ in the interval $[a,a+T]$, provided that
its distance to the left end point of the interval is at most $l$.
That is,
$$
L_T(f,a)=\inf\{t\in[a,a+T]: f(t)=h, t\leq a+l\}.
$$
Then $L_T$ is a local intrinsic location functional. However, its
``natural'' extension, $L(f,[a,b]):=\inf\{t\in[a,b]: f(t)=h, t\leq
a+l\}$ is not an intrinsic location functional. To see this,
notice that the existence of such a location in an interval with
length $T$ does not guarantee its existence for all the larger
intervals containing it, since the location may fail to remain
close enough to the left end point when the interval expands.}
\end{example}

It turns out that despite the large variety covered by the concept
of local intrinsic location functional, they all correspond to the
idea of taking the maximal element in a random set, ordered
according to some specific rule.

\begin{theorem}\label{thm:order.representation2}
Let $H$ be defined as before. A mapping $L_T=L_T(f,a)$ from
$H\times \mathbb R$ to ${\mathbb R}\cup\{\infty\}$ is a local
intrinsic location functional with related length $T$, if and only
if
\begin{enumerate}
\begin{item}
$L_T(\cdot,a)$ is measurable for $a\in \mathbb R$;
\end{item}
\begin{item}
For each function $f\in H$, there exists a subset of $\mathbb{R}$
denoted as $S(f)$ and a partial order $\preceq$ on it, satisfying:
\begin{enumerate}
\item For any $c\in\mathbb R$, $S(f)=S(\theta_c f)+c$; \item For
any $c\in\mathbb R$ and any $t_1,t_2\in S(f)$, $t_1\preceq t_2$
implies $t_1-c\preceq t_2-c$ in $S(\theta_c f)$,
\end{enumerate}
such that for any $a\in {\mathbb R}$, either $S(f)\cap
[a,a+T]=\phi$, in which case $L_T(f,a)=\infty$, or $L_T(f,a)$ is the maximal element in
$S(f)\cap [a,a+T]$ according to $\preceq$.
\end{item}
\end{enumerate}
\end{theorem}

\begin{proof}
It is easy to check that the measurability of $L_T(\cdot,a)$ for $a\in\mathbb R$ and the existence of such an ordered set $S(f)$ for $f\in H$ guarantee that $L_T$ is a local intrinsic location functional. For the other direction, let $L_T$ be a local intrinsic location functional with related length $T$. For each path $f$, define a set
$$
S(f)=\{t\in {\mathbb R}: t=L_T(f,a) \text{ for some } a\in{\mathbb R}\}.
$$
Thus $S(f)$ is the set of all the points which is chosen as the
location for some interval with length $T$. From now on we fix the
function $f$ and simplify the notation $S(f)$ as $S$. We introduce
the following partial binary relation on $S$. For two points
$x,y\in S$, say $x\preceq_0 y$ if and only if there exists an
interval $I_{x,y}=[a_{x,y},a_{x,y}+T]$, such that $x,y\in I_{x,y}$
and $L_T(f,a_{x,y})=y$. In another word, $x\preceq_0 y$ if and
only if some interval with length $T$ containing both of them
``chooses'' $y$ rather than $x$ to be its corresponding location.
Then we complete $\preceq_0$ by taking the smallest transitive
binary relation containing it, denoted as $\preceq$. We claim that
such defined $\preceq$ is actually a partial order on the set $S$.

The reflexivity is clear: by definition, $x\preceq x, \forall x\in
S$. The transitivity is also guaranteed by construction. Therefore
the only thing left is to check the antisymmetry: if $x\preceq y$
and $y\preceq x$, then $x=y$. To this end, firstly notice that the
construction of the binary relation $\preceq_0$ guarantees that it
is always antisymmetric before being extended to $\preceq$. That
is, $x\preceq_0 y$ and $y\preceq_0 x$ implies $x=y$. Now assume
$x\neq y$, $x\preceq y$ and $y\preceq x$, then there is a loop:
$x=t_0\preceq_0 t_1\preceq_0 ... \preceq_0 y=t_n\preceq_0
t_{n+1}\preceq_0 ...\preceq_0 t_{n+m-1}\preceq_0 t_{n+m}=x$ for
some positive integers $m,n$, and points
$t_0,t_1,...,t_{n+m-1},t_{n+m}=t_0$ satisfying $|t_{i+1}-t_i|\leq
T$ for any $i=0,...,n+m-1$.

To deal with this loop, notice that we have the proposition below,
which states that if two points within a distance no larger than $T$ have a relation $\preceq$ between them, then there must be
a direct relation given by $\preceq_0$. They can not be only
related through a chain of ``$\preceq_0$'' via other points.

\begin{lemma}\label{lemma:direct}
Let the relations $\preceq_0$ and $\preceq$ be as defined above.
Then $t_1\preceq t_2$ and $|t_2-t_1|\leq T$ imply $t_1\preceq_0
t_2$ or $t_2\preceq_0 t_1$.
\end{lemma}

\begin{proof}
Proof by contradiction. Without loss of generality, assume there
are two points $t_1, t_2\in S$, $t_1<t_2$, $t_2-t_1\leq T$, there
exist points $s_0,s_1,...,s_n,s_{n+1}$ such that $s_0=t_1\preceq_0
s_1\preceq_0 ... \preceq_0 s_n\preceq_0 s_{n+1}=t_2$, however,
there is no direct relation given by $\preceq_0$ between $t_1$ and
$t_2$. That is, every interval with length $T$ containing the
interval $[t_1, t_2]$ have neither $t_1$ nor $t_2$ as its
corresponding location. Since $t_1\in S$, there is $a\in{\mathbb
R}$, such that $L_T(f,a)=t_1$. The interval $[a,a+T]$ can not
include $t_2$, otherwise $t_2\preceq_0 t_1$. Therefore $a+T<t_2$.
Consider $L_T(f,t_2-T)$. Because
$L_T(f,a)=t_1\in[a,a+T]\cap[t_2-T,t_2]$, the condition (4) in the
definition of local intrinsic location functional rules out the
possibility that $L_T(f,t_2-T)=\infty$ or $L_T(f,t_2-T)\in
[a,a+T]\cap[t_2-T,t_2]$. Thus $L_T(f,t_2-T)\in (a+T,t_2]\subseteq
(t_1, t_2]$. It can not be $t_2$ either since then $t_1\preceq_0
t_2$. As a result, $L_T(f,t_2-T)\in(t_1, t_2)$. Denote
$L_T(f,t_2-T)$ by $t_3$. Then $t_3\in S$ and by definition
$t_1\preceq_0 t_3$ and $t_2\preceq_0 t_3$.

Consider the intervals $[s_j,s_{j+1})$ for $j=0,...,n$ which
satisfies $s_j<s_{j+1}$. Clearly, their union covers the interval
$[t_1,t_2)$, therefore also the point $t_3$. Assume
$t_3\in[s_k,s_{k+1})$. There are two cases. Case 1: $s_{k+1}\leq
t_2$. Since $s_k\preceq_0 s_{k+1}$, there is a real number $a_1$,
such that $s_k\in[a_1,a_1+T]$ and $L_T(f,a_1)=s_{k+1}$. Similarly,
since $t_2\preceq_0 t_3$, there is a real number $a_2$ such that
$t_2\in[a_2,a_2+T]$ and $L_T(f,a_2)=t_3$. However $s_{k+1}\leq
t_2$ implies that both $s_{k+1}$ and $t_3$ are in the interval
$[a_1,a_1+T]\cap[a_2,a_2+T]$, thus contradicting with the
definition of local intrinsic location functional. Case 2:
$s_{k+1}>t_2$. In this case, notice that $t_2\in S$, so there
exists $a_3$ such that $L_T(f,a_3)=t_2$. However, since
$\preceq_0$ is antisymmetric, $t_2\preceq_0 t_3$ implies that
$t_3\npreceq_0 t_2$, so $a_3>t_3$. Now both $t_2$ and $s_{k+1}$
are in the interval $[a_1,a_1+T]\cap[a_3, a_3+T]$, yet $L_T$ gives
out different locations, contradiction again. To conclude, the
assumption at the beginning of the proof can not hold, and the
lemma is proved.
\end{proof}

Now we turn back to the loop and prove the following result: there
exist $i_1, i_2, i_3\in\{0,...,n+m-1\}$, such that
$t_{i_1}\preceq_0 t_{i_2}\preceq_0 t_{i_3}\preceq_0 t_{i_1}$.
Consider the rightmost point in the set
$\{t_i\}_{i=0,...,n+m-1}$, denoted as $t_j:=\max_{i=0}^{n+m-1}
t_i$. Notice that $t_{j-1}<t_j$, $t_{j+1}<t_j$, therefore
$|t_{j+1}-t_{j-1}|<T$, and $t_{j-1}\preceq_0 t_j\preceq_0 t_{j+1}$
(define ${t_{-1}}=t_{n+m-1}$). By lemma \ref{lemma:direct} there
is a relation $\preceq_0$ between $t_{j+1}$ and $t_{j-1}$. If
$t_{j+1}\preceq_0 t_{j-1}$, we already have a loop with three
terms as desired. If $t_{j-1}\preceq_0 t_{j+1}$, then consider the
set $\{t_i\}_{i=0,...,n+m-1, i\neq j}$. It is also a loop as the
set $\{t_i\}_{i=0,...,n+m-1}$ by which we started, now with one
less term. An iteration of this procedure finally decreases the
size of the set to 3, so we find a loop with 3 terms again.

The existence of a loop with 3 terms, however, contradicts with
the definition of the relation $\preceq_0$. To see this, without
loss of generality, suppose that we have $t_1<t_2<t_3$ satisfying
$t_1\preceq_0 t_2\preceq_0 t_3\preceq_0 t_1$.
This means that there exists $a,b\in{\mathbb R}$, such that
$t_1,t_3\in [a,a+T]$ and $L_T(f,a)=t_1$, $t_1,t_2\in[b,b+T]$ and
$L_T(f,b)=t_2$. However, the fact that $t_1,t_2\in
[a,a+T]\cap[b,b+T]$, yet $L_T(f,a)$ and $L_T(f,b)$ are not equal
contradicts with the definition of local intrinsic location
functional.

In total, we have seen that a loop of relation $\preceq_0$,
therefore also $\preceq$, is not possible. Thus the antisymmetry
is proved. The relation $\preceq$ is a partial order. Finally, it
is clear by the construction of the partially ordered set
$(S(f),\preceq)$ that either $S(f)\cap [a,a+T]=\phi$, in which
case $L_T(f,a)=\infty$, or $L_T(f,a)\in S(f)\cap[a,a+T]$, in which
case $s\preceq L_T(f,a)$ for all $s\in S(f)\cap[a,a+T]$.
\end{proof}

\begin{remark}\label{r:generalization}{\rm
The partial order in the theorem has the special property that
there exists a unique maximal element over any interval with
length $T$. In this sense it behaves like a total order. Indeed,
by order extension principle, the partial order $\preceq$ can
always be extended to a total order on $S(f)$, and it is clear
that we can do it in a shift-invariant way, so that the resulting
total order also satisfies the conditions in Theorem
\ref{thm:order.representation2}. Nonetheless, here we would like
to keep $\preceq$ a partial order for generality.}
\end{remark}

A similar reasoning allows us to derive the ordered set representation for intrinsic location functionals.

\begin{corollary}\label{cor:order.representation}
Let $H$, $\mathcal I$ be defined as before. A mapping $L=L(f,I)$
from $H\times \mathcal I$ to ${\mathbb R}\cup\{\infty\}$ is an
intrinsic location functional if and only if
\begin{enumerate}
\begin{item}
$L(\cdot,I)$ is measurable for $I\in \mathcal I$;
\end{item}
\begin{item}
For each function $f\in H$, there exists a partially ordered
subset of ${\mathbb R}$, denoted as $(S(f),\preceq_1)$,
satisfying:
\begin{enumerate}
\item For any $c\in\mathbb R$, $S(f)=S(\theta_c f)+c$; \item For
any $c\in\mathbb R$ and any $t_1,t_2\in S(f)$, $t_1\preceq_1 t_2$
implies $t_1-c\preceq_1 t_2-c$ in $S(\theta_c f)$,
\end{enumerate}
such that for any $I\in {\mathcal I}$, either $S(f)\cap I=\phi$,
in which case $L(f,I)=\infty$, or $L(f,I)$ is the maximal element in $S(f)\cap I$ according to $\preceq_1$.
\end{item}
\end{enumerate}
\end{corollary}

\begin{proof}
Again, it is routine to check the ``if'' direction. For the other direction, define $S(f):=\{t\in \mathbb R: L(f,I)=t \text{ for some } I
\in{\mathcal I}\}$ and the binary relation $\preceq_1$ on $S(f)$: $x\preceq_1 y$ if and only if there exists an interval $I\in\mathcal I$ such that $x,y\in I$ and $L(f,I)=y$. The argument goes through in the same way, and is actually simpler, since such defined $\preceq_1$ is now directly a partial order.
\end{proof}

\begin{example}{\rm
Let $H$ be the space of all upper semi-continuous functions on
$\mathbb R$. The location of the path supremum
$\tau_{f,I}:=\inf\{t\in I: f(t)=\sup_{s\in I}f(s)\}, f\in H, I\in
{\mathcal I}$ is an intrinsic location functional. It corresponds
to an ordered set $(S(f),\preceq)$, where $S(f)=S^1(f)\cup
S^2(f)$, $S^1(f)$ is the union of the set of local maxima of $f$,
and $S^2(f):=\{t\in {\mathbb R}: t=\sup_{s\in[t-T,t]}(f(s)) \text{
or } t=\sup_{s\in[t,t+T]}(f(s))\}$. ``$\preceq$'' is firstly
ordered by the value of the function at the points and in case of
a tie, inversely ordered by the location (that is, the
locations on the left receive high orders).}
\end{example}

\begin{example}{\rm
Let $H$ be the space of all continuous functions on $\mathbb R$.
The first hitting time of a level $l$ over an interval $I$,
defined by $T^l_{f,I}:=\inf\{t\in I: f(t)=l\}$ is also an
intrinsic location functional. The ordered set $(S(f),\preceq)$ is
now given by $S(f)=f^{-1}(l)$ and the inverse order on the real
line.}
\end{example}

It is clear that the partially ordered random set representation
of a local intrinsic location functional or an intrinsic location
functional can not be unique, since one can always add irrelevant
points to $S(f)$ and assign them very low orders, so that the
added points are actually never chosen as the location for any
interval. However, there exists a unique minimal representation,
as indicated by the proof of Theorem
\ref{thm:order.representation2}.

\begin{corollary}\label{cor:minimal.representation}
Let $L$ be a local intrinsic location functional (resp. intrinsic
location functional) with path space $H$. There exists a partially
ordered set $(S(f),\preceq)$ for each function $f\in H$,
satisfying the conditions in Theorem
\ref{thm:order.representation2} (resp. Corollary
\ref{cor:order.representation}), such that for any other partially
ordered set $(S'(f),\preceq ')$ also satisfying the same
conditions,
$$
S(f)\subseteq S'(f)
$$
and
$$
s_1, s_2\in S(f), s_1\preceq s_2 \text{ implies } s_1\preceq ' s_2 \text{ in } S'(f).
$$
\end{corollary}

The proof is easy and omitted. Notice that we do not only know the existence of the minimal representation, it is actually straightforward to write it down explicitly. For a local intrinsic location functional $L_T$ with related length $T>0$ and $f\in H$, $S(f)=\{t: L_T(f,a)=t \text{ for some } a\in\mathbb R\}$, and $\preceq$ is the smallest partial order such that $s_1\preceq s_2$ for all $s_1, s_2$ satisfying $s_1\in[a,a+T]$ and $L_T(f,a)=s_2$ for some $a\in\mathbb R$. Similarly, for an intrinsic location functional $L$ and $f\in H$, $S(f)=\{t: L(f,I)=t \text{ for some } I\in\mathcal I\}$, and $\preceq$ is given by $s_1\preceq s_2$ if $s_1\in I$ and $L(f,I)=s_2$ for some $I\in\mathcal I$.

\section{Extension and restriction}

The ordered set representation provides powerful tools for us to clarify the link between intrinsic location functional and local intrinsic location functional. The theorem below shows that a local intrinsic location functional is
``almost'' just a ``local'' version of an intrinsic location functional.

We call a mapping $L$ from $H\times {\mathcal I}$ to ${\mathbb
R}\cup \{\infty\}$ a ``pre-intrinsic location functional'', if it
satisfies all the defining properties in Definition \ref{de:ILF}
except for the measurability condition (1). In another word, a
pre-intrinsic location functional becomes an intrinsic location
functional once it is measurable for all compact intervals $I\in\cI$.

\begin{theorem}\label{th:extension}
Let $L$ be an intrinsic location functional. Then for any $T>0$,
\begin{equation}\label{eq:extension}
L_T(f,a):=L(f,[a,a+T])
\end{equation}
is a local intrinsic location functional. Conversely, let $L_T$ be
a local intrinsic location functional. Then there exists a
pre-intrinsic location functional $L$, such that (\ref{eq:extension})
holds for all $f\in H$ and $a\in\mathbb R$.
\end{theorem}

\begin{proof}
The fact that a restricted intrinsic location functional is a
local intrinsic location functional can be easily checked either
by their definitions or by the ordered set representation. For the
other direction, suppose we have a local intrinsic location
functional $L_T$, with the partially ordered set $(S(f),\preceq)$
for each $f\in H$. By the order extension principle,
$(S(f),\preceq)$ can always be extended, in a shift-invariant way,
to a totally ordered set $(S(f),\preceq_1)$, which is, of course,
a special partially ordered set. Define $L(f,I)$ for any $I\in
\mathcal I$ by taking the maximal element in $I$ of $S(f)$
according to $\preceq_1$: $L(f,I)\in S(f)$ and $s\preceq_1 L(f,I)$
for all $s\in S(f)\cap I$, then by Corollary
\ref{cor:order.representation} such defined $L$ is a pre-intrinsic
location functional.
\end{proof}

Notice, however, that we have not touched the measurability issue
and claimed that each local intrinsic location functional
necessarily has an intrinsic location functional extension. The
problem of measurability is highly nontrivial and in general, the
measurability of a local intrinsic location functional for
intervals with a single fixed length may not be enough to
guarantee the measurability of its extensions with all different
interval lengths. Instead, we prove the following result, which
shows that there always exists an intrinsic location functional
which agrees almost surely with the given local intrinsic location
functional for any stationary process in the interior of any
interval with the fixed length.

\begin{proposition}\label{prop:almost.same}
Let $L_T:H\times{\mathbb R}\to{\mathbb R}\cup\{\infty\}$ be a local intrinsic location functional with related length $T$. Then there exists an intrinsic location functional $L:H\times \cI\to{\mathbb R}\cup\{\infty\}$, such that for any $a\in\mathbb R$ and stationary process $\BX$ with paths in $H$,
$$
\mathbb P [L_T(\BX,a)\neq L(\BX,[a,a+T]),
$$
$$
L_T(\BX,a)\in(a,a+T)\text{ or }L(\BX,[a,a+T])\in(a,a+T) ]=0.
$$
\end{proposition}

Before we go to the proof of Proposition \ref{prop:almost.same},
let us first look at a useful lemma.

\begin{lemma}\label{l:monotonicity}
Let $L_T$ be a local intrinsic location functional defined on
$H\times \mathbb R$. Then
\begin{enumerate}
\begin{item}
For any $f\in H$, any $a<b$ such that $L_T(f,a)\neq\infty$ and
$L_T(f,b)\neq\infty$, $L_T(f,a)\leq L_T(f,b)$.
\end{item}
\begin{item}
If $L_T(f,a)=L_T(f,b)=t\neq\infty$, then $L_T(f,c)=t$ for any
$c\in [a,b]$.
\end{item}
\begin{item}
If $a<b$, $b-a\leq T$ and $L_T(f,a)=L_T(f,b)=\infty$, then
$L_T(f,c)=\infty$ for all $c\in[a,b]$.
\end{item}
\end{enumerate}
\end{lemma}

\begin{proof}
Suppose for some $a<b$, $L_T(f,b)< L_T(f,a)<\infty$. Then both
$L_T(f,a)$ and $L_T(f,b)$ are in the interval
$[a,a+T]\cap[b,b+T]$. However, by the definition of local
intrinsic location functional, this implies that they must be
equal. Thus the first claim of the proposition is proved. Now
assume $L_T(f,a)=L_T(f,b)=t\neq\infty$. Then
$t\in[a,a+T]\cap[b,b+T]=[b,a+T]\neq\phi$. For any $c\in[a,b]$,
$[c,c+T]\supseteq [b,a+T]$, hence $t\in[a,a+T]\cap[c,c+T]$. By
definition of local intrinsic location functional,
$L_T(f,c)\neq\infty$. Then by the first claim of the proposition,
$t=L_T(f,a)\leq L_T(f,c)\leq L_T(f,b)=t$. Therefore $L_T(f,c)=t$
as well. Finally, if $a<b$, $b-a\leq T$, then for any $c\in[a,b]$,
$[c,c+T]\subset [a,a+T]\cup[b,b+T]$. If
$L_T(f,a)=L_T(f,b)=\infty$, then by Theorem
\ref{thm:order.representation2}, $[a,a+T]\cap S(f)=[b,b+T]\cap
S(f)=\phi$, where $S(f)$ is a set of points corresponding to
$L_T$. As a result, $[c,c+T]\cap S(f)=\phi$, which, going back to
$L_T$, means that $L_T(f,c)=\infty$.

\end{proof}

\begin{proof}[Proof of Proposition \ref{prop:almost.same}]
For any function $f\in H$, define the sets
$$
S_1(f):=\{t\in {\mathbb R}: \exists (x,y)\subset {\mathbb R}, \text{ s.t. } L_T(f,a)=t, \forall a\in (x,y)\},
$$
$$
S_2(f):=\{t\in{\mathbb R}\setminus S_1(f): L_T(f,t)=t \text{ or } L_T(f,t-T)=t\}
$$
and $S'(f)=S_1(f)\cup S_2(f)$.

On $S'(f)$ assign a binary relation $\preceq_0$: $t_1\preceq_0
t_2$ if and only if $|t_2-t_1|<T$ and there exists a real number
$a$ satisfying $t_1,t_2\in[a,a+T]$ such that $L_T(f,a)=t_2$.
Notice that the set $S'(f)$ is a subset of the set we constructed
in the proof of Theorem \ref{th:extension}, and $\preceq_0$ is
also a restriction of the corresponding binary relation that we
saw before. As a result, one can again extend $\preceq_0$ to a
smallest partial order, still denoted by $\preceq$.

For function $f\in H$ and a compact interval $I$, define $L(f,I)$
to be the first element in $S'(f)$ which is maximal in $I$:
$$
L(f,I)=\inf\{t\in S'(f)\cap I: t'\in S'(f)\cap I \text{ and }
t\preceq t' \text{ implies } t'=t\}.
$$
We can denote the set on the right hand side of the definition
above, namely, the set of all the maximal in $I$ points in
$S'(f)$, by $M_{f,I}$. Then $L(f,I)$ is simply $\inf(M_{f,I})$,
with the tradition that $\inf(\phi)=\infty$. Indeed, this way of
choosing the first maximal element is equivalent to assigning an
additional order among the maximal elements according to their
location, with the left receiving the higher order and the right
lower. The resulting new order will then satisfy all the
conditions listed in Corollary \ref{cor:order.representation},
which assures that such defined $L(f,I)$ is a pre-intrinsic
location functional. Thus all that is left is to check the
measurability.

Fix $I=[a,b]$ with $|I|=b-a>T$ and $f\in H$. The event $\{L(f,I)\leq s\}$ is
$\{a\in M_{f,I}\}$ if $s=a$, $\{a\in M_{f,I}\}\cup
\{M_{f,I}\cap(a,s]\neq\phi\}$ if $s\in (a,b)$, and $\{a\in
M_{f,I}\}\cup \{M_{f,I}\cap (a,s)\neq\phi\}\cup \{b\in M_{f,I}\}$ if $s=b$.
Therefore it suffices to verify the measurability for each of
these sets.

\begin{lemma}\label{l:series}
Let $I=[a,b]$, $b-a>T$ and $t\in(a,b)$, then $t\in M_{f,I}$ if and only if
for some sequences $\{t_{1n}\}_{n=1,2,...}$ and
$\{t_{2n}\}_{n=1,2,...}$ such that $t_{1n}\to t$ and $t_{2n}\to t$
as $n\to\infty$, $L_T(f, a\vee (t_{1n}-T))=L_T(f,(b-T)\wedge
t_{2n})=t$ holds for $n=1,2,...$.
\end{lemma}

\begin{proof}

Firstly assume that $t\in M_{f,I}\cap (a,b)$. If $a\leq t-T$, then
for any $s\in(t,(t+T)\wedge b)$, $[s-T,s]\subset (a,b)$, and $t\in
(s-T,s)$. By the maximality of $t$ under the partial order
$\preceq$, $L_T(f,s-T)=t$. Therefore we only need to take
$\{t_{1n}\}_{n=1,2,...}$ a decreasing sequence converging to $t$
with $t_{11}<(t+T)\wedge b$ to have $L_T(f,t_{1n}-T)=t$. If $a >
t-T$, then the maximality implies that $L_T(f,a)=t$. Combining
these two cases, there always exist $\{t_{1n}\}_{n=1,2,...}$ such that $L_T(f, a\vee (t_{1n}-T))=t$. Symmetrically we
have $L_T(f,(b-T)\wedge t_{2n})=t$ for some $\{t_{2n}\}_{n=1,2,...}$.

The case where $t\notin S'(f)$ being trivial, now suppose $t\in S'(f)\cap (a,b)$ but $t\notin M_{f,I}$. Then
there exists $s\in (t-T,t+T)\cap [a,b]$ such that $t\preceq_0 s$.
Without loss of generality, assume that $s<t$. Then for any $r\in
[t-T,s]$, $L_T(f,r)\neq t$, since otherwise $s\preceq_0 t$.
Therefore there does not exist a sequence
$\{t_{1n}\}_{n=1,2,...}$, such that $L_T(f, a\vee (t_{1n}-T))=t$.
The lemma is proved.
\end{proof}

For any $x,y$ such that $a\leq x<y\leq b$, denote by $E_{I}(x,y)$
the event $L_T(f,a\vee (y-T))=L_T(f,(b-T)\wedge x)\neq \infty$.
For $r,s\in(a,b)$ and $m=1,2,...$, define event
$E_{I,m}(r,s)=\bigcup_{i=1}^{2^m-1}
E_{I}(r+\frac{(i-1)(s-r)}{2^m},r+\frac{(i+1)(s-r)}{2^m})$.
Consider the set
$$
\begin{array}{rcl}
E(I,r,s) & := & \bigcup_{n=1}^\infty \bigcap_{m=n}^\infty E_{I,m}(r,s)\\
& = & \bigcup_{n=1}^\infty \bigcap_{m=n}^\infty \bigcup_{i=1}^{2^m-1} E_{I}(r+\frac{(i-1)(s-r)}{2^m},r+\frac{(i+1)(s-r)}{2^m}).
\end{array}
$$
It is clearly measurable. Suppose there is a point $t\in(r,s)$ in $M_{f,I}$. For any $m$ large enough, let $i'$ be an index satisfying $t\in(r+\frac{(i'-1)(s-r)}{2^m},r+\frac{(i'+1)(s-r)}{2^m})$. Then event $E_{I}(r+\frac{(i'-1)(s-r)}{2^m},r+\frac{(i'+1)(s-r)}{2^m})$ holds. Consequently $E_{I,m}(r,s)$ holds hence $E(I,r,s)$ also holds. Thus $\{M_{f,I}\cap (r,s)\neq\phi\}\subseteq E(I,r,s)$. On the other hand, suppose $E(I,r,s)$ is realized. Then for all $m$ large enough, $E_{I}(r+\frac{(i-1)(s-r)}{2^m},r+\frac{(i+1)(s-r)}{2^m})$ holds for some $i=1,...,2^m-1$. Denote by $J_m$ the set of indices $i=1,...,2^m-1$ for which $E_I(r+\frac{(i-1)(s-r)}{2^m},r+\frac{(i+1)(s-r)}{2^m})$ holds, and $B_m:=\bigcup_{i\in J_m}[r+\frac{(i-1)(s-r)}{2^m},r+\frac{(i+1)(s-r)}{2^m}]$. It is easy to check by definition that $B_m$ is a decreasing sequence of closed sets, thus there exists some point $t\in [r,s]$ which is covered by infinite members in $\{B_m\}_{m=0,1,...}$, therefore also infinite number of intervals forming $B_m, m=0,1,...$. Let $\{I_{m_j}=[a_{m_j},b_{m_j}]\}_{j=1,2,...}$ be such a sequence always covering $t$. Notice that $a_{m_j}\to t$ and $b_{m_j}\to t$ as $j\to\infty$. Moreover, $E_I(a_{m_j},b_{m_j})$ holds for all $j=1,2,...$ by construction. Thus by Lemma \ref{l:series} $t\in M_{f,I}$. Thus we have
$$
\{M_{f,I}\cap (r,s)\neq\phi\}\subseteq E(I,r,s)\subseteq \{M_{f,I}\cap[r,s]\neq\phi\},
$$
which implies
$$
E(I,r,s)\cup\{r\in M_{f,I}\}\cup\{s\in M_{f,I}\}=\{M_{f,I}\cap[r,s]\neq\phi\}.
$$
It is easy to check that $\{r\in M_{f,I}\}$ and $\{s\in M_{f,I}\}$ are measurable. $\{r\in M_{f,I}\}$, for example, can only happen if $r\in M_{f,I}\cap S_1(f)$, which is then equivalent to $\cup_{n=1}^\infty\cap_{m=n}^\infty E_I(r-\frac{1}{m},r+\frac{1}{m})$. As a result, $\{M_{f,I}\cap [r,s]\neq\phi\}$ is measurable for any $r,s$ in the interior of $[a,b]$. It is then trivial to see the measurability of $\{M_{f,I}\cap (a,s]\neq\phi\}$ for $s\in(a,b)$ or $\{M_{f,I}\cap (a,s)\neq\phi\}$ for $s=b$ by taking a countable union. The case for the two endpoints $a$ and $b$ can be checked directly. The measurability of $a\in M_{f,I}$, for instance, is verified once we observe that $a\in M_{f,I}\cap S_1(f)$ if and only if there exists a sequence $\{s_n\}_{n=1,2...}$, such that $s_n\uparrow a$ and $L_T(f,s_n)=a$ for $n=1,2,...$. $a\in M_{f,I}\cap S_2(f)$, of course, if and only if $L_T(f,a)=a$.

For the case of $I=[a,b]$ with $|I|= T$, the key is to notice that $L(f,I)=t\in(a,b)$ if and only if there exists a positive integer $n$ such that $L(f,I_{1n})=t$ or $L(f,I_{2n})=t$, where $I_{1n}=[a-\frac{1}{n}, b]$, $I_{2n}=[a,b+\frac{1}{n}]$, and $L(f,I_{1n})$ and $L(f,I_{2n})$ are defined as above for $|I|>T$. Thus for any $s\in(a,b)$,
$$
\{L(f,I)\in[a,s])\}=\{a\in M_{f,I}\}\cup\left(\bigcup_{n=1}^\infty\{L(f,I_{in})\in(a,s], i=1 \text{ or } 2\}\right)
$$
is measurable. The cases with $s=a$ or $s=b$ are not much different from before.

Finally if $I=[a,b]$ with $|I|<T$, $L(f,I)=t\in(a,b)$ is
equivalent to the existence of three points $x,y\in {\mathbb Q}$
and $z\in({\mathbb Q}\cap [b-T,a])\cup\{a,b-T\}$, such that
$L_T(f,x)=L_T(f,y)=L_T(f,z)=t$. It is not difficult to check this
equivalence. Intuitively, the existence of $x$ and $y$ assures
that $t\in S'(f)$, while the existence of $z$ guarantees the
maximality of $t$ in $I$. The countability of the rational set
then leads to the measurability. We skip the details.

Combining the three cases proves the measurability of $L(\cdot, I)$ for any compact interval $I$, as desired. $L$ is thus an intrinsic location functional. The last thing in the proof is therefore to show the relationship between $L_T(\BX,a)$ and $L(\BX,[a,a+T])$ claimed in the proposition.

Let $\BX$ be any stationary process with paths in $H$. Firstly,
assume $L(\BX,[a,a+T])=t\in(a,a+T)$ but $L_T(\BX,a)\neq t$. Then
$L_T(\BX,a)\notin S'(\BX)$, since otherwise $t$ and $L_T(\BX,a)$
are both in $S'(\BX)$, $|t-L_T(\BX,a)|<T$ and by definition of
$\preceq_0$, $t\preceq_0 L_T(\BX,a)$, contradicting the maximality
of $L(\BX,[a,a+T])$. By the same reasoning, if
$L_T(\BX,a)\in(a,a+T)$ but $L(\BX,[a,a+T])\neq L_T(\BX,a)$ then
$L_T(\BX,a)\notin S'(\BX)$. Together, we have
$$
\{L_T(\BX,a)\neq L(\BX,[a,a+T]),
$$
$$
L_T(\BX,a)\in(a,a+T)\text{ or }L(\BX,[a,a+T])\in(a,a+T)\}
$$
$$
\subseteq \{L_T(\BX,a)\notin S'(\BX)\}.
$$

Notice that if $L_T(\BX, a)=a$ or $L_T(\BX,a)=a+T$, then
$L_T(\BX,a)\in S'(\BX)$ automatically. By the definition of
$S'(X)$, $L_T(\BX,a)\notin S'(\BX)$ if and only if $L_T(\BX,a)\in
(a,a+T)$ and $L_T(\BX,a)\neq L_T(\BX,b)$ for any $b\neq a$, which
is equivalent to $L_T(\BX,a)\neq L_T(\BX,b)$ for any $b\neq a$,
$b\in \mathbb Q$ by Lemma \ref{l:monotonicity}. Thus
$\{L_T(\BX,a)\notin S'(\BX)\}$ is measurable. Now we show that
${\mathbb P}(L_T(\BX,a)\notin S'(\BX))=0$. Assume ${\mathbb
P}(L_T(\BX,a)\in (a,a+T)\setminus S'(\BX))>0$. Then there exists
$\Delta>0$, such that ${\mathbb P}(L_T(\BX,a)\in
(a+\Delta,a+T-\Delta)\setminus S'(\BX))=:\delta>0$. Take
$\epsilon<\Delta/(\lfloor 1/\delta\rfloor)$ and compact intervals
$I_i=[a+i\epsilon, a+i\epsilon+T]$ for $i=0,1,...,\lfloor
1/\delta\rfloor$, where ``$\lfloor\cdot\rfloor$'' refers to the
largest integer smaller or equal to the argument. By construction,
for any $i,j=0,1,...,\lfloor 1/\delta\rfloor$, $I_i\cap I_j\supset
[a+i\epsilon+\Delta, a+i\epsilon+T-\Delta]\cup[a+j\epsilon+\Delta,
a+j\epsilon+T-\Delta]$. This, however, implies that the events
$E_i:=\{L_T(\BX,a+i\epsilon)\in
(a+i\epsilon+\Delta,a+i\epsilon+T-\Delta)\setminus S'(\BX)\}$ must
be disjoint for different $i$. Otherwise, suppose $E_i$ and $E_j$
holds for some $i<j$. Then since both $L_T(\BX,a+i\epsilon)$ and
$L_T(\BX,a+j\epsilon)$ are in the intersection of $I_i$ and $I_j$,
they must be equal. Lemma \ref{l:monotonicity} then implies that
$L_T(\BX,a')=L_T(\BX,a+i\epsilon)$ for all
$a'\in[a+i\epsilon,a+j\epsilon]$. This contradicts with $E_i$,
which requires that $L_T(\BX,a+i\epsilon)\notin S'(\BX)$. By
stationarity, ${\mathbb P}(E_i)={\mathbb P}(L_T(\BX,a)\in
(a+\Delta,a+T-\Delta)\setminus S(\BX))=\delta, i=0,1,...,\lfloor
1/\delta\rfloor$. Then
$$
{\mathbb P}(\bigcup_{i=0}^{\lfloor 1/\delta\rfloor} E_i)=\delta\cdot \left(\lfloor 1/\delta\rfloor+1 \right)>1,
$$
which clearly shows a contradiction. As a result,
${\mathbb P}(L_T(\BX,a)\notin S'(\BX))=0$ and the proof of the
proposition is complete.
\end{proof}

The importance of Proposition \ref{prop:almost.same} resides in
the fact that most of the distributional properties of intrinsic
location functionals proved in \cite{samorodnitsky:yi:2013b} can
now be transformed automatically to local intrinsic location
functionals. In particular, local intrinsic location functionals
always satisfy the total variation constraints. Thus the
equivalence between the stationarity, the total variation
constraints and the shift invariance of the distributions can be
extended to local intrinsic location functionals.

\begin{corollary}\label{cor:equiv.lilf.fixT}
Let $\BX$ be a stochastic process with continuous paths. Let ${\mathcal L}_{loc,T}$ be the set of all local intrinsic location functionals in ${\mathcal C}(\mathbb R)$ with related length $T$. Then the followings are
equivalent:
\begin{enumerate}
\begin{item}
The process $\BX$ is stationary.
\end{item}
\begin{item}
For any $T>0$, any local intrinsic location functional $L_T\in{\mathcal
L}_{loc,T}$, the distribution of $L_T(\BX,a)-a$ does not depend on $a$.
\end{item}
\begin{item}
For any $T>0$, any local intrinsic location functional $L_T\in{\mathcal L}_{loc,T}$
and any $a\in\mathbb R$, $L_T(\BX,a)$ admits a density function
$f_{\BX,a,T}(t)$ in $(a,a+T)$, which satisfies the total variation
constraint on $[a,a+T]$.
\end{item}
\end{enumerate}
\end{corollary}

\begin{remark}{\rm
A closer examination of the proof of the equivalence theorem in
\cite{samorodnitsky:yi:2013b} shows that the length of the
interval does not play any crucial role in the proof of the
equivalence between (1) and (2). As a result, $(2)$ in Corollary
\ref{cor:equiv.lilf.fixT} is also equivalent to:

$(2')$ For a fixed $T>0$, any local intrinsic location functional
$L_T\in{\mathcal L}_{loc,T}$, the distribution of $L_T(\BX,a)-a$
does not depend on $a$.}
\end{remark}

To sum up, while the equivalence between the stationarity and the
total variation constraints of the intrinsic location functionals
have been established in \cite{samorodnitsky:yi:2013b}, we just
extended this result to local intrinsic location functionals,
which is more generally defined compared to intrinsic location
functionals. Moreover, the local intrinsic location functionals
are further identified with the shift-compatible ordered sets of
points $(S(\cdot),\preceq)$ on $\mathbb R$ as path functionals.
Such an identification provides a particularly convenient way to
define local intrinsic location functionals.

We complete this section by the following corollary, which
examines the relation between intrinsic location functionals and
local intrinsic location functionals, from the perspective of the
partially ordered sets they correspond to. The proof is easy and omitted.

\begin{corollary}
Let $H$, $\mathcal I$ be defined as before. Let $L$ be an intrinsic location functional, then $L_T: H\times {\mathbb R}\to{\mathbb R}\cup\{\infty\}$ defined by $L_T(f,a)=L(f,[a,a+T])$ is a local intrinsic location functional for each $T>0$. If $(S(\cdot),\preceq_1)$ and $(S_T(\cdot), \preceq_T)$ are the minimal ordered random set representations for $L$ and $L_T$ respectively, then for $f\in H$, $S_T(f)\subseteq S(f)$, and $t_1,t_2\in S_T(f), t_1\preceq_T t_2$ implies $t_1\preceq_1 t_2$ in $S(f)$.

On the other hand, let $\{L_T\}_{T>0}$ be a family of local intrinsic location functionals, with minimal ordered random set representations $\{(S_T(\cdot),\preceq_T)\}_{T>0}$. Then there exists an intrinsic location functional $L$ such that $L(f,I)=L_{b-a}(f,a)$ for any $I=[a,b]\in \mathcal I$ and any $f\in H$, if and only if their exists a partially ordered random set $(S(\cdot),\preceq_1)$ satisfying the same properties as in condition (2) in Corollary \ref{cor:order.representation}, such that for any $T>0$, $f\in H$, $S_T(f)\subseteq S(f)$, and $t_1,t_2\in S_T(f), t_1\preceq_T t_2$ implies $t_1\preceq_1 t_2$ in $S(f)$.
\end{corollary}

\section{Path characterization}

Let $L_T$ be a local intrinsic location functional with related
length $T$. Given any $f\in H$, define $g(x):=L_T(f,x)-x, \forall
x\in \mathbb R$. Thus $g(x)$ is the distance between $L_T$
and the starting point $x$ of the interval of interest. Notice that since $L_T$ can take value infinity, $g(x)$ can also be infinity. The following result gives
out a characterization of the function $g$. In another
word, it answers the question how we can tell whether a random location is a local intrinsic location functional by looking at the change of its place in the interval as the interval shifts over the real line.

We call a partition satisfying certain property the ``roughest'',
if all the other partitions satisfying the same property is a
refinement of the given partition.

\begin{theorem}\label{thm:path}
Let $L_T$ be a local intrinsic location functional discussed
before, and $g$ be the function defined above. Then for any $f\in
H$, there exists a roughest partition of the real line by
intervals (the intervals can be degenerated, and the boundaries of
the intervals can be open or closed), such that for any member
$I=(a,b),(a,b], [a,b)$ or $[a,b]$ of this partition, exactly one
of the followings is true.

(1) $b-a\leq T$, and $g(x)=d-x$ for some $d\in[b,a+T]$ and all
$x\in I$ .

(2) $g(x)=\infty$ for all $x\in I$.

Moreover, if $g(a)\neq T$ (resp. $g(b)\neq 0$), then
$\lim_{x\uparrow a}g(x)=0$ (resp. $\lim_{x\downarrow b}g(x)=T)$.
If $I$ is open on $a$ (resp. $b$), then $g(a)=0$ (resp. $g(b)=T$).

On the other hand, let $L_T$ be a mapping from $H\times \mathbb R$
to ${\mathbb R}\cup\{\infty\}$ such that $L_T(\cdot, a)$ is
measurable for any $a\in \mathbb R$, and $L_T(f,a)=L_T(\theta_c
f,a-c)+c$ for any $a,c\in\mathbb R$. If for any function $f\in H$,
there always exists a partition $P$ of the real line by intervals
satisfying the properties listed above, then $L_T$ is a local
intrinsic location functional with related length $T$.
\end{theorem}

Roughly speaking, Theorem \ref{thm:path} tells us that the
function $g$ consists of linear pieces with slope $-1$ and
intervals with value $\infty$. The pieces are combined together
following the rule that when the interval $[x,x+T]$ shifts along
the real line, a location can ``disappear'' in the interior of the
interval only if it is replaced by another location appearing at
the right endpoint $x+T$. Symmetrically, a location can only
``appear'' in the interior of the interval only if it is replacing
another location disappearing at the left endpoint $x$. The actual
scenario is a little bit more complicated, since both the replaced
and replacing ``location'' can be indeed the limit of a sequence
of locations, where comes the limits in the formulation of the
theorem.

\begin{proof}
Let $L_T$ be a local intrinsic location functional with related
length $T$. By Theorem \ref{thm:order.representation2} and Remark
\ref{r:generalization}, for each $f\in H$, there is a set
$S(f)\subseteq \mathbb R$ and a partial order $\preceq$ on it,
satisfying $S(\theta_c f)=S(f)-c$ and $t_1\preceq t_2$ in $S(f)$
if and only if $t_1-c\preceq t_2-c$ in $S(\theta_c f)$ for any
$c\in\mathbb R$, such that $L_T(f,x)$ is the unique maximal
element by $\preceq$ in $S(f)\cap [x,x+T]$ for any $f\in H$ and
any $x\in\mathbb R$, provided that it exists. For any fixed
$x\in\mathbb R$, there are two cases. Case 1: $S(f)\cap [x,
x+T]=\phi$. In this case define $a=\sup\{S(f)\cap(-\infty, x)\}$
and $b=\inf\{S(f)\cap(x+T,\infty)\}-T$. Then $a,b$ are clearly the
two boundaries of the largest interval containing $x$ on which
$L_T(f,\cdot)=\infty$. Notice that it is possible to have $a=b$,
in which case the interval becomes degenerate. Case 2: $S(f)\cap
[x,x+T]\neq \phi$. In this case define
$a=\max\{L_T(f,x)-T,\sup\{y\in{\mathbb R}: y\in S(f), y<L_T(f,x),
L_T(f,x)\preceq y\}\}$ and $b=\min\{L_T(f,x) ,\inf\{y\in{\mathbb
R}: y\in S(f), y>L_T(f,x), L_T(f,x)\preceq y\}-T\}$. Then
$L_T(f,x)$ will remain the same when and only when $x$ moves
between $a$ and $b$. That is, $L_T(f,y)=L_T(f,x)$ for $y\in I$,
$I=[a,b],[a,b),(a,b]$ or $(a,b)$, whether the boundary is closed
or open being determined by which one is larger/smaller in the max
and min in the definition of $a$ and $b$, and whether the supremum
and infimum are achieved by a single point or only by a sequence
of points. As a result, for any $y\in I$,
$g(y)=L_T(f,y)-y=L_T(f,x)-y=d-y$ where $d:=L_T(f,x)\in\cap_{y\in
I} [y,y+T]\subseteq [b,a+T]$. Thus case 2 corresponds to scenario
(1) and case 1 corresponds to scenario (2) in Theorem
\ref{thm:path}.

Next we check the combination rule, that is, the sentence below
the two scenarios in the theorem. Firstly assume $g(a)\neq T$.
Hence either $g(a)<T$ or $g(a)=\infty$. If $g(a)<T$, consider
$g(x)$ for $x\in(a-T+g(a),a)$. Notice that $x+T>a+g(a)=L_T(f,a)$.
However, $L_T(f,x)$ can not be equal to $L_T(f,a)$, since
otherwise by Lemma \ref{l:monotonicity} $a$ will not be the left
endpoint of a largest interval on which $g(\cdot)$ is linear.
Hence $L_T(f,x)\in[x,x+T]\backslash [a,a+T]=[x,a)$. Since $x$ can
be arbitrarily close to $a$, this implies $g(x)\to 0$ as
$x\uparrow a$.

The argument for the possibility $g(a)=\infty$ is similar. For any
$x<a$, $L_T(f,x)\in[x,a)$ or $L_T(f,x)=\infty$. The last instance,
however, is not possible when $x>a-T$, since otherwise by Lemma
\ref{l:monotonicity} the interval $I$ will not be the largest
interval on which $g$ is $\infty$. Thus $L_T(f,x)\in[x,a)$, which
then implies that $g(x)\to 0$ as $x\uparrow a$.

In the same spirit, we can show that if $I$ is open at $a$, then $g(a)=0$. Assume it is not the case. Then $g(a)=\infty$ or $0<g(a)\leq T$. If $g(a)=\infty$ and $g(x)$ is also infinity on $(a,b)$ or $(a,b]$, the maximality of the interval $I$ is violated; if $g(x)=d-x$ for any $x\in I$ and some $d\in [b,a+T]$, then $L_T(f,x)=d \in[a,a+T]$, which contradicts with $L_T(f,a)=g(a)+a=\infty$ according to the definition of local intrinsic location functional. Hence we must have $0<g(a)\leq T$. Consider a point $s\in(a,\min(a+g(a),b))$. $L_T(f,s)=d\in[b,a+T]\subseteq [s,a+T]$. However $L_T(f,a)=a+g(a)\in[s,a+T]$, thus $L_T(f,s)=L_T(f,a)$, contradicting with the openness of $I$ on $a$. Therefore both of the two possibilities fail and $g(a)$ must take value $0$.

Now let us turn to the other direction of the proof. The
measurability and shift invariance are already given. The value
range $L_T(f,a)\in[a,a+T]\cup\{\infty\}$ for any $f\in H$ and
$a\in\mathbb R$ is easy to check. It remains to show condition (4) in
Definition \ref{defn:lilf}. Before we proceed, notice that the
combination rule leads to the following fact:

\begin{lemma}\label{l:slow.decrease}
Let $g: {\mathbb R}\to{\mathbb R}\cup\{\infty\}$ be a function
satisfying the combination rule. Then for $x,y\in{\mathbb R}, x<y$
satisfying $g(x)\neq \infty$ and $g(y)\neq\infty$, $g(x)-g(y)\leq
y-x$. The equality holds if and only if $x$ and $y$ are in the
same maximal interval in Theorem \ref{thm:path}. Equivalently, let
$L_T(t)=g(t)+t$ for $t\in\mathbb R$, then for $x,y\in{\mathbb R},
x<y$ satisfying $L_T(x)\neq \infty$ and $L_T(y)\neq\infty$,
$L_T(x)\leq L_T(y)$. The equality holds if and only if $x$ and $y$
are in the same maximal interval in Theorem \ref{thm:path}.
\end{lemma}

The proof of this lemma is easy and omitted here.

Let $y_1<y_2$ be two arbitrary points on the real line. We can assume
that $y_2-y_1\leq T$, since otherwise the condition $L_T(f,y_2)\in
[y_1,y_1+T]$ can never be satisfied. There are two cases. Case 1:
$y_1$ and $y_2$ are in the same interval $I$, on which $g(x)=d-x$
or $g(x)=\infty$. Clearly, in this case $L_T(f,y_1)=L_T(f,y_2)$.
Case 2: $y_1$ and $y_2$ are not in the same interval. Say, $y_2\in
I_2$ and $y_1\notin I_2$, where $I_2=[a_2,b_2], [a_2,b_2),
(a_2,b_2]$ or $(a_2,b_2)$ is the largest interval containing $y_2$
on which $g(x)=d-x$ or $g(x)=\infty$. Notice that $L_T(f,y_1)\neq
L_T(f,y_2)$, since otherwise the monotonicity implies that
$L_T(f,x)=L_T(f,y_2)$ for all $x\in[y_1,y_2]$, contradicting with
the assumption that $I$ is the largest interval. Our goal is
therefore to prove that in this case,
$L_T(f,y_2)\in[y_1,y_1+T]\cap[y_2,y_2+T]=[y_2,y_1+T]$ implies
$L_T(f,y_1)\in[y_1,y_2)$.

Assume that $L_T(f,y_2)\in[y_2,y_1+T]$. Firstly, $L_T(f,y_1)$ can not be infinity. Otherwise, let $I_1$ be
the largest interval containing $y_1$ on which the location takes
value $\infty$. By the combination rule $\lim_{y\downarrow
b_1}g(y)=T$, where $b_1$ is the right endpoint of $I_1$. $b_1\geq
y_1$ so $y_2-b_1\leq y_2-y_1$. Meanwhile
$L_T(f,y_2)\in[y_2,y_1+T]$ implies that $g(y_2)=L_T(f,y_2)-y_2\leq
y_1+T-y_2$, thus $\lim_{y\downarrow b_1}g(y)-g(y_2)=T-g(y_2)\geq
y_2-y_1$. If equality actually holds for both this inequality and
the previous one, then $y_1=b_1$, and $\lim_{y\downarrow
b_1}g(y)-g(y_2)=y_2-y_1$, hence also $\lim_{y\downarrow
y_1}g(y)-g(y_2)=y_2-y_1$. By Lemma \ref{l:slow.decrease}, $y_1\geq
a_2$, where $a_2$ is the left endpoint of the maximal interval
$I_2$ containing $y_2$. Since $y_1\notin I_2$, $y_1=a_2$ and $I_2$
is open at $y_1$. However, by combination rule, this implies that
$g(y_1)=T\neq\infty$, contradiction. Thus the two inequalities can
not be equalities at the same time. As a result,
$\lim_{y\downarrow b_1}g(y)-g(y_2)>y_2-b_1$, which, however,
contradicts with Lemma \ref{l:slow.decrease}. Thus $L_T(f,y_1)\neq
\infty$.

Next, notice that $\lim_{y\downarrow
a_2}L_T(f,y)=L_T(f,y_2)\in[y_2,y_1+T]$. If $g(a_2)=T$, then
$g(a_2)-\lim_{g\downarrow a_2}g(y)\geq 0=\lim_{y\downarrow a_2}-a_2$.
According to Lemma \ref{l:slow.decrease}, this can only happen if
$a_2\in I_2$. However, $y_1\leq a_2 \leq L_T(f,a_2)\leq y_1+T$ and
$T=g(a_2)=L_T(f,a_2)-a_2$ implies that $y_1=a_2$. Together we have
$y_1\in I_2$, contradiction. Thus $g(a_2)\neq T$. Therefore by
combination rule, $\lim_{y\uparrow a_2}g(y)=0$. If $a_2<y_2$, then
by the monotonicity of $L_T(f,\cdot)$ given by Lemma
\ref{l:slow.decrease}, $L_T(f,y_1)\leq \lim_{y\uparrow
a_2}L_T(f,y)=a_2\in[y_1,y_2)$. Therefore we only need to consider
the case where $a_2=y_2$. Suppose that in this case
$L_T(f,y_1)\geq a_2=y_2$. By the monotonicity of $L_T(f,\cdot)$ and
the fact that $\lim_{y\uparrow a_2}L_T(f,y)=a_2=y_2$, $b_1$ must
be equal to $y_2$, where $b_1$ is the right endpoint of the
maximal interval $I_1$ containing $y_1$, and the inequality above can only be an equality. As a result, $I_1$
is open at $y_2$, therefore $g(y_2)=T$ by combination rule. This contradicts with the
assumption that $L_T(f,y_2)=g(y_2)+y_2\in[y_2,y_1+T]$. To
conclude, $L_T(f,y_1)<y_2$, hence $L_T(f,y_1)\in[y_1,y_2)$. The
second direction of Theorem \ref{thm:path} is therefore proved.

\end{proof}

\noindent\textbf{Acknowledgment.}

The author expresses grateful thanks to his Ph.D. advisor, Gennady Samorodnitsky, for his direction and helpful suggestions.

\end{document}